\documentclass[final,leqno]{siamltex}

\usepackage{amsfonts,amssymb} 
\usepackage{amsmath} 
\usepackage{color} 
\usepackage{graphicx} 
\usepackage{epsfig,mathrsfs} 
\usepackage[bf,SL,BF]{subfigure} 
\usepackage{fancyhdr} 
\usepackage{CJK} 
\usepackage{caption} 
\usepackage{wrapfig} 
\usepackage{cases} 
\usepackage{bm}
\usepackage{algorithm}
\usepackage{algorithmic}
\newtheorem{example}{Example}[section] 

\title{A Multigrid method for nonlocal  problems: non-diagonally dominant  Toeplitz-plus-tridiagonal systems
\thanks{This work was supported by NSFC 11601206. }}

\author{Minghua Chen\thanks{Corresponding author. School of Mathematics and Statistics, Gansu Key Laboratory of Applied Mathematics and Complex Systems,
 Lanzhou University, Lanzhou 730000, P.R. China  (Email: chenmh@lzu.edu.cn)}
\and Sven-Erik Ekstr\"om \thanks{
Department of Information Technology, Division of Scientific Computing, Uppsala University - ITC, L ägerhyddsv. 2, hus 2,  P.O. Box 337, SE-751 05, Uppsala, Sweden
(Email: sven-erik.ekstrom@it.uu.se) }
 \and Stefano Serra-Capizzano \thanks{
Department of Science and High Technology, University of Insubria, Via Valleggio 11, 22100 Como, Italy $\&$ Department of Information Technology, Division of Scientific Computing, Uppsala University - ITC, L ägerhyddsv. 2, hus 2,  P.O. Box 337, SE-751 05, Uppsala, Sweden
(Email: stefano.serrac@uninsubria.it, stefano.serra@it.uu.se) }
 }

\begin{document}

\maketitle

\begin{abstract}
The nonlocal problems have been used to model very different applied scientific phenomena, which involve  the  fractional  Laplacian when one looks  at   the  L\'{e}vy processes and stochastic interfaces. This paper deals with the nonlocal problems  on a bounded domain, where the stiffness matrices of the resulting systems are  Toeplitz-plus-tridiagonal  and   far from being diagonally dominant, as it occurs when dealing with  linear  finite element approximations. By exploiting a  weakly diagonally dominant   Toeplitz property of the stiffness matrices, the optimal convergence of the two-grid method is well established [Fiorentino and Serra-Capizzano,  {\em SIAM J. Sci. Comput.},  {17} (1996), pp. 1068--1081; Chen and Deng, {\em SIAM J. Matrix Anal. Appl.}, {38} (2017),  pp. 869--890]; and  there are still questions about best ways to define coarsening and interpolation operator when the stiffness matrix is far from being weakly diagonally dominant [St\"{u}ben,  {\em  J. Comput. Appl. Math.},  {128} (2001), pp. 281--309]. In this work, using spectral indications from our analysis of the involved matrices, the simple (traditional) restriction operator and prolongation operator are employed in order to handle general algebraic systems which are {\em neither   Toeplitz nor weakly diagonally dominant} corresponding to the fractional Laplacian kernel and the constant kernel, respectively. We focus our efforts on  providing  the detailed proof of the convergence of the two-grid method for  such situations. Moreover, the convergence of the  full multigrid is also discussed with the constant kernel.   The numerical experiments are performed to verify the convergence with only  $\mathcal{O}(N \mbox{log} N)$ complexity by the fast Fourier transform, where $N$ is the number of the grid points.
\end{abstract}



\pagestyle{myheadings}
\thispagestyle{plain}
\markboth{M. H. CHEN, S.  EKSTR\"OM, AND S. SERRA-CAPIZZANO}{MULTIGRID METHOD FOR NONLOCAL  PROBLEMS}

\section{Introduction}
The nonlocal diffusion problems have been used to model very different scientific phenomena occurring in various applied fields, for example in biology, particle systems, image processing, coagulation models, mathematical finance, etc \cite{Andreu:10}.
When one looks at the L\'{e}vy processes and  stochastic interfaces, the nonlocal operator that appears naturally is the fractional Laplacian \cite{Andreu:10,Bertoin:96,Zoia:07}. Recently, the nonlocal volume-constrained diffusion problems,  the  so-called nonlocal model for distinguishing the nonlocal diffusion problems, attracted the wide interest of scientists \cite{Aksoylu:14,Du:12}, where the linear scalar peridynamic model can be considered as a special case  \cite{Du:12,Silling:00}.  In particular, it is  pointed out  that  the fractional Laplacian for anomalous diffusion are special cases of the nonlocal model \cite{D'Elia:13,Du:12}. In the literature, we have already a lot of important proposals for numerically solving nonlocal problems or nonlocal models. For example, a  finite difference-quadrature scheme  for the fractional Laplacian  have been derived in  \cite{Huang:14}. Finite element approximations for the fractional Laplacian \cite{Acosta:17,D'Elia:13} or tempered fractional Laplacian  have been discussed in \cite{Zhang:18}. A
fast conjugate gradient  Galerkin method   has been used for solving efficiently  the resulting system arising from  a peridynamic model \cite{Wang:12}.
In this work, we especially focus our efforts on efficient multigrid method (MGM), by providing the strict convergence proof for  nonlocal problems.

Multigrid method  are among the most efficient iterative methods for solving large scale systems of equations, arising from the discretization of partial differential equations (PDEs) \cite{Bramble:87,Hackbusch:01,Xu:17}.
Using the  regularity and  finite element approximation theory,  the convergence estimates of the V-cycle MGM have  been proved for the elliptic PDEs \cite{Bank:85,Bramble:87,Bramble:91} and the  fractional problems \cite{CNOJ:16,CBQW:18,Jiang:15}. By using the compact notion of symbol and its basic analytical features, the V-cycle optimal convergence has been derived for the case of  multilevel linear system whose coefficient matrices belong to the circulant, Hartely, or $\tau$ algebras or  to the Toeplitz class \cite{Arico:07,Arico:04,Bolten:15,Serra:02}. It directly  tackles the  stiffness matrix of the resulting algebraic system, which can often be derived directly from the underlying matrices, without any reference to the grids.
In general, when considering the discretization of PDEs/FDEs, it is hard to make a theoretically concise statement  \cite{Xu:97,Xu:17}, but it is possible to give a rather simple analysis \cite{Xu:97} for a very special one dimensional  elliptic PDEs and two dimensional case \cite{CDS:17}. However, it is still not at all easy for the dense stiffness matrices  \cite{Arico:07,Arico:04,Bolten:15,CD:17}, unless we can reduce the problem to the Toeplitz setting and we know the symbol, its zeros, and their orders  \cite{Serra:02}. Instead we will focus our attention on first answering such a question for a two-level setting, since  it is useful from a theoretical point of view as the first step to study the MGM convergence usually begins from the two-grid method (TGM) \cite{Pang:12,Ruge:87,Saad:03,Xu:17}.

As is well known, a theoretical analysis of the related two-grid method is given in terms of the algebraic multigrid theory considered in  \cite{Ruge:87}. For solving Toeplitz systems, the convergence of the TGM on the first level is  proved for the so-called band $\tau$ matrices \cite{Fiorentino:96} and  a  complete analysis of convergence of the TGM is given for elliptic Toeplitz and PDEs Matrix-sequences \cite{Serra:02}.
For a class of   weakly diagonally dominant Toeplitz matrices, the uniform convergence   of the TGM are theoretically obtained \cite{CCS:98} and are extended to nonlocal operators \cite{CD:15,CWCD:14,Pang:12}.
For Toeplitz-plus-diagonal systems,  the numerical behaviour of MGM  have  been discussed in  \cite{NST:05} and the preconditioned  Krylov subspace methods including  conjugate gradient method have been proposed
in \cite{DMS:16,MDDM:17,PKNS:14,Wang:12}. It should be noted that the proof technique is different and in fact the related proofs do not rely on the diagonal dominance   \cite{Arico:07,Arico:04,CCS:98,Serra:02}. In reality,  the two-grid optimality (and sometimes the V-cycle optimality) is proven for Toeplitz matrices $T_n(f)$ where $f$ is nonnegative and has isolated zeros: the proof and the algorithm depend on the position and order of such zeros and most of such matrices are far from being diagonally dominant  \cite{Arico:07,Arico:04,Serra:02}.
For Toeplitz-plus-tridiagonal systems  arising from nonlocal problems, to the best of our knowledge,  no fast MGMs have been developed and no convergence analysis has been provided;
moreover,  there are still questions for MGM when is far from being weakly diagonally dominant \cite{Stuben:01}. In this paper, first we give a structural and spectral analysis of the underlying matrices. Then, based on the latter study, the simple (traditional) restriction operator and prolongation operator are employed in order to handle general algebraic systems: here the stiffness matrices of the resulting systems are  Toeplitz-plus-tridiagonal  and   far from being diagonally dominant, respectively, corresponding to the fractional Laplacian kernel and the constant kernel. We focus on  providing  the detailed proof of the convergence of the two-grid method for  such situations. Moreover, the convergence of the  full multigrid, i.e.,  recursive application of the TGM procedure \cite{CCS:98,Fiorentino:96,Serra:02}, is also discussed with the constant kernel. The performed numerical experiments show the effectiveness of the MGM with only  $\mathcal{O}(N \mbox{log} N)$ complexity by the fast Fourier transform, where $N$ is the number of the grid points.

The outline of this paper is as follows. In the next section, we derive the algebraic systems  for the  nonlocal problems arising from    linear  finite element approximations, we study their structural and spectral features, and we introduce the MGM algorithms. In Section 3, we study the uniform  convergence estimates of the TGM for the considered nonlocal problem with the fractional Laplacian kernel. Convergence of the TGM and  full MGM with  a constant kernel case is analyzed in Section 4. To show the effectiveness of the presented schemes, results of numerical experiments are reported in Section 5. Finally, we conclude the paper with some remarks.

\section{Preliminaries: numerical scheme}
In this section, we derive  the numerical discretization for nonlocal diffusion problems with the fractional Laplacian kernel and  the constant kernel, respectively.
Let $\Omega$ be a finite bar in $\mathbb{R}$.  The nonlocal  operator  is used in  the time-dependent nonlocal  diffusion problem   \cite{Andreu:10,Bates:06,Mengesha:13}
\begin{equation*}
\left\{ \begin{split}
 u_t(x,t) - \mathcal{L}u(x,t) &=f(x,t),    &  &  x \in \Omega,\, t>0,\\
                        u(x,0)&=u_0(x),    & &   x \in \Omega,\\
                        u(x,t)&=0,         & &   x \in \mathbb{R}\setminus \Omega,
 \end{split}
 \right.
\end{equation*}
 and its steady-state counterpart
\begin{equation}\label{2.1}
\left\{ \begin{split}
  - \mathcal{L} u(x)      &=f(x), &  & x \in \Omega,\\
                 u(x)&=0,        &    &x \in \mathbb{R}\setminus \Omega.
 \end{split}
 \right.
\end{equation}
Here, the  nonlocal operator $\mathcal{L}$  is defined by
\begin{equation*}
\begin{split}
\mathcal{L}u(x)=\int_{\Omega}[u(y)-u(x)]J(|x-y|)dy ~~\forall x \in \Omega,
\end{split}
\end{equation*}
where $J$ is a radial probability density with a nonnegative  symmetric dispersal kernel.

From \cite{Mengesha:13}, we obtain
\begin{equation*}
\begin{split}
(-\mathcal{L}u,v)
&=\int_{\Omega}v(x)\int_{\Omega}[u(x)-u(y)]J(|x-y|)dydx\\
&=\frac{1}{2}\int_{\Omega}\int_{\Omega}[u(y)-u(x)][v(y)-v(x)]J(|x-y|)dydx.
\end{split}
\end{equation*}
The   energy space associated with (\ref{2.1}) is defined as
\begin{equation*}
  \mathcal{S}(\Omega)=\{ u\in L^2(\Omega) \big| \int_{\Omega}\int_{\Omega}[u(y)-u(x)]^2J(|x-y|)dydx <\infty  \}
\end{equation*}
 and $\mathcal{S}_0(\Omega)=\{ u\in \mathcal{S}(\Omega)\big| u=0 ~{\rm in}~  \mathbb{R}\setminus \Omega \}.$

Formally, we can define a bilinear form $a(u,v):\mathcal{S}_0(\Omega) \times \mathcal{S}_0(\Omega)\rightarrow \mathbb{R} $ by
\begin{equation*}
  a(u,v)=\frac{1}{2}\int_{\Omega}\int_{\Omega}[u(y)-u(x)][v(y)-v(x)]J(|x-y|)dydx,
\end{equation*}
and the weak formulation of (\ref{2.1}) is expressed as follows: finding $u\in \mathcal{S}_0(\Omega)$ such that
\begin{equation}\label{2.2}
  a(u,v)=(f,v) ~~\forall v \in \mathcal{S}_0(\Omega).
\end{equation}

It was proved that the bilinear form $a(u,v)$ is coercive and bounded on the nonconventional Hilbert space $\mathcal{S}_0(\Omega)$ and second-order convergence can be expected for linear finite element, with sufficiently smooth functions \cite{Du:12,Mengesha:13}. Let $\Omega=(0,b)$ with the mesh points $x_i=ih$, $h=b/N$ and  $u_{i}$ as the numerical approximation of $u(x_i)$ and  $f_{i}=f(x_i)$. Denote $I_i=((i-1)h,ih)$ for $1 \leq i \leq N$, and the  piecewise linear  basis function is
\begin{equation*}
\phi_i(x)=\left\{ \begin{array}{lll}
 \displaystyle\frac{x-x_{i-1}}{h}, &  x \in [x_{i-1}, x_{i}],\\
 \displaystyle\frac{x_{i+1}-x}{h}, & x \in [x_{i}, x_{i+1}],\\
 \displaystyle 0, &  {\rm otherwise}
 \end{array}
 \right.
\end{equation*}
with $i = 1,2\ldots N-1$, and
\begin{equation*}
\phi_0(x)=\left\{ \begin{array}{lll}
 \displaystyle\frac{x_{1}-x}{h}, & x \in [x_{0}, x_{1}],\\
 \displaystyle 0, &  {\rm otherwise},
 \end{array}
 \right.
 ~~~~~~~~
\phi_{N}(x)=\left\{ \begin{array}{lll}
 \displaystyle\frac{x-x_{N-1}}{h}, & x \in [x_{N-1}, b],\\
 \displaystyle 0, &  {\rm otherwise}.
 \end{array}
 \right.
\end{equation*}

Let $\mathcal{S}_0^h(\Omega) \subset\mathcal{S}_0(\Omega)$ be the finite element space consisting of  piecewise linear  polynomials $\phi_i(x)$ with respect to the uniform  mesh.
The finite element approximation of the variational problem (\ref{2.2}) is expressed in accordance with the continuous setting: find $u_h\in \mathcal{S}_0^h(\Omega)$ such that
\begin{equation}\label{2.3}
  a(u_h,v_h)=(f,v_h) ~~\forall v_h \in \mathcal{S}_0^h(\Omega).
\end{equation}
Using   $u_h=\sum_{i=1}^{N-1}u_i\phi_i(x)$, we can rewrite (\ref{2.3}) as
\begin{equation}\label{2.4}
 \sum_{i=1}^{N-1} a(\phi_i,\phi_j)u_i=(f,\phi_j), ~~~ j=1,2,\ldots N-1,
\end{equation}
which is representable as a linear system of equations
\begin{equation}\label{2.5}
A_hu_h=f_h, ~~(A_h)_{i,j}=a_{i,j}=a(\phi_i,\phi_j),~~{\rm and}~~(f_h)_j=(f,\phi_j).
\end{equation}
Here, $u_h=[u_1,u_2,\ldots,u_{N-1}]^{\rm T}$ and the matrix $A$ is known as the stiffness matrix of the nodal basis $\{\phi_i\}_{i=1}^{N-1}$. More concretely,
\begin{equation}\label{2.6}
(A_h)_{i,j}=a_{i,j}=a(\phi_i,\phi_j)=\int_{x_{i-1}}^{x_{i+1}}\phi_i(x)\left[\int_{\Omega}[\phi_j(x)-\phi_j(y)]J(|x-y|)dy\right]dx.
\end{equation}

In this paper,  we mainly focus on two types of the    classical  kernel  functions for (\ref{2.1}),
i.e., the fractional Laplacian kernel \cite{Andreu:10,Du:12} and  the constant kernel \cite{Aksoylu:14,Tian:13}. More general kernel types \cite{Aksoylu:14,Andreu:10,Du:12,Tian:13}  can be similarly studied.

\subsection{Discretization scheme for (\ref{2.1}) with fractional Laplacian kernel: a Toeplitz-plus-tridiagonal  system} \label{ssec:1}

In this subsection, we choose   the fractional Laplacian kernel  $J(|x|)=C_\alpha/|x|^{1+\alpha}$ with $\alpha \in (1,2)$, then  (\ref{2.1}) reduces to
the following   integral version   fractional Laplacian:
\begin{equation}\label{2.7}
\left\{ \begin{split}
 (-\Delta)^{\frac{\alpha}{2}}u(x):=C_\alpha\int_{\Omega}\frac{u(x)-u(y)}{|x-y|^{1+\alpha}}dy       &=f(x), &  & x \in \Omega,\\
                 u(x)&=0,        &    &x \in \mathbb{R}\setminus \Omega
 \end{split}
 \right.
\end{equation}
with
$$C_\alpha=\frac{\alpha2^{\alpha-1}\Gamma(\frac{1+\alpha}{2})}{\pi^{1/2}\Gamma(1-\alpha/2)}=\kappa_\alpha \frac{-\alpha}{\Gamma(1-\alpha)}>0,~~ \kappa_\alpha=\frac{-1}{2\cos(\alpha\pi/2)}>0,~~\alpha \in (1,2).$$
It is should be noted that there exist several   equivalent definitions of the fractional Laplacian, agreing with the  space of appropriately smooth functions \cite{Kwasnicki:17}.

From (\ref{2.6}) and (\ref{2.7}),  the entries $a_{i,j}$ ($j\geq i+2$) of the  matrix are given by
\begin{equation*}
\begin{split}
& a_{i,j}=a(\phi_i,\phi_j)\\
& =C_\alpha \int_{x_{i-1}}^{x_i}\frac{x-x_{i-1}}{h}\left[\int_{x_{j-1}}^{x_j}\frac{0-\frac{y-x_{j-1}}{h}}{(y-x)^{1+\alpha}} dy+\int_{x_j}^{x_{j+1}}\frac{0-\frac{x_{j+1}-y}{h}}{(y-x)^{1+\alpha}} dy\right]dx\\
&\quad + C_\alpha\int_{x_i}^{x_{i+1}}\frac{x_{i+1}-x}{h}\left[\int_{x_{j-1}}^{x_j}\frac{0-\frac{y-x_{j-1}}{h}}{(y-x)^{1+\alpha}} dy+\int_{x_j}^{x_{j+1}}\frac{0-\frac{x_{j+1}-y}{h}}{(y-x)^{1+\alpha}} dy\right]dx\\
&=- C_\alpha h^{1-\alpha}\Bigg[\int_{0}^{1}t\int_{0}^{1}\frac{s}{(j-i+s-t)^{1+\alpha}}dsdt+\int_{0}^{1}t\int_{0}^{1}\frac{s}{(j-i+2-s-t)^{1+\alpha}}dsdt\\
&\quad+\int_{0}^{1}t\int_{0}^{1}\frac{s}{(j-i-2+s+t)^{1+\alpha}}dsdt+\int_{0}^{1}t\int_{0}^{1}\frac{s}{(j-i-s+t)^{1+\alpha}}dsdt \Bigg].
\end{split}
\end{equation*}
Similarly, we can obtain  $a_{i,i}$ and  $a_{i,i+1}=a_{i+1,i}$.  By calculation, it is easy to get
\begin{equation}\label{2.8}
  A_h=\frac{\kappa_\alpha}{h^{\alpha-1}\Gamma(4-\alpha)}B_h,
\end{equation}
where  the entries of the stiffness matrix $(B_h)_{i,j}=b_{i,j}$ are explicitly given by
\begin{equation}\label{2.9}
\begin{split}
b_{i,i}
&= \left(8-2^{4-\alpha}   \right)+\widetilde{b}_{i,i},~i=1,2,\ldots N-1,\\
b_{i,i+1}
&=b_{i+1,i}
= \left(-7-3^{3-\alpha} +2^{5-\alpha}  \right)+\widetilde{b}_{i,i+1},~i=1,2,\ldots N-1,\\
b_{i,j}
&=-\left( m+2 \right)^{3-\alpha}+4\left( m+1 \right)^{3-\alpha}\\
&\quad-6m^{3-\alpha}+4\left( m-1 \right)^{3-\alpha}-\left( m-2 \right)^{3-\alpha},~m=|j-i|\geq 2
\end{split}
\end{equation}
with
\begin{equation*}
\begin{split}
\widetilde{b}_{i,i}
=& 2\left[ \left(i+1   \right)^{3-\alpha}-2(3-\alpha)i^{2-\alpha} -\left(i-1   \right)^{3-\alpha}\right] \\
 &+2\left[ \left(N-i+1   \right)^{3-\alpha}-2(3-\alpha)(N-i)^{2-\alpha} -\left(N-i-1   \right)^{3-\alpha}\right],\\
\widetilde{b}_{i,i+1}
=& -2\left[ \left(i+1   \right)^{3-\alpha}-i^{3-\alpha}\right]+\left(3-\alpha\right)\left[ \left(i+1   \right)^{2-\alpha}+i^{2-\alpha}\right] \\
 &\!\!-2\left[ \left(N-i   \right)^{3-\alpha}\!-\!\left(N-i -1  \right)^{3-\alpha}\right]\!+\!\left(3-\alpha\right)\left[ \left(N-i   \right)^{2-\alpha}\!+\!\left(N-i -1  \right)^{2-\alpha}\right].
\end{split}
\end{equation*}
In fact, the stiffness matrix $A_h$ (or equivalently $B_h$) is a symmetric  Toeplitz-plus-tridiagonal matrix, i.e.,
\begin{equation*}
B_h=T_h+E_h.
\end{equation*}
Here
\begin{equation*}
T_h=\left [ \begin{matrix}
                      c_0           &      c_{1}             &      \cdots         &       c_{N-2}       \\
                      c_{1}         &      c_{0}              &      \cdots         &       c_{N-3}        \\
                     \vdots         &      \vdots             &      \ddots         &        \vdots            \\
                     c_{N-2}        &      c_{N-3}            &      \cdots         & c_{0}
 \end{matrix}
 \right ]
\end{equation*}
with
\begin{equation*}
\begin{split}
c_{0}
&= 8-2^{4-\alpha}  ,\\
c_{1}
&
= -7-3^{3-\alpha} +2^{5-\alpha} ,\\
c_{m}
&=-\left( m+2 \right)^{3-\alpha}+4\left( m+1 \right)^{3-\alpha}
-6m^{3-\alpha}+4\left( m-1 \right)^{3-\alpha}-\left( m-2 \right)^{3-\alpha},~m\geq 2,
\end{split}
\end{equation*}
and
\begin{equation*}
\begin{split}
E_h= \left [ \begin{matrix}
a_1      &    b_1           &            &                    \\
b_1     &    a_2            &  b_2        &           &         \\
       &   \ddots        &  \ddots    &   \ddots  &          \\
       &                 &  b_{N-3}       &       a_{N-2}   &      b_{N-2}    \\
       &                 &            &       b_{N-2}   &     a_{N-1}
\end{matrix}
\right ]_{(N-1)\times(N-1)}
\end{split}
\end{equation*}
with
\begin{equation*}
\begin{split}
&a_i=\widetilde{b}_{i,i},~~1\leq i\leq N-1,\\
&b_i=\widetilde{b}_{i,i+1},~~1\leq i\leq N-2.
\end{split}
\end{equation*}

\subsection{Discretization scheme for (\ref{2.1}) with constant kernel:  a non-diagonally dominant system}  \label{ssec:2}
For simplicity, we choose   the constant kernel  $J(|x|)=1$, then  (\ref{2.1}) reduces to the following  pseudo-differential equation
\begin{equation}\label{2.10}
\left\{ \begin{split}
   \int_{\Omega}\left[u(x)-u(y)\right]dy      &=f(x), &  & x \in \Omega,\\
                 u(x)&=0,        &    &x \in \mathbb{R}\setminus \Omega.
 \end{split}
 \right.
\end{equation}
Using   (\ref{2.6}) and (\ref{2.10}), it is immediate to obtain
\begin{equation}\label{2.11}
  A_h=h^2B_h
\end{equation}
with
\begin{equation*}
B_h = \left[ \begin{matrix}
\frac{2N}{3}-1    &  \frac{N}{6}-1 &    -1 &   \cdots  &   -1 \\
\frac{N}{6}-1     & \frac{2N}{3}-1   &   \ddots  &   \ddots   &   \vdots \\
-1 & \ddots  &   \ddots  &   \ddots   &   -1   \\
 \vdots    &\ddots  &   \ddots  &   \frac{2N}{3}-1       &   \frac{N}{6}-1    \\
-1     &  \cdots    &   -1  &     \frac{N}{6}-1      &   \frac{2N}{3}-1
 \end{matrix}
 \right ]_{(N-1)\times(N-1)},
\end{equation*}
i.e.,   the entries of the stiffness matrix $(B_h)_{i,j}=b_{i,j}$ are explicitly given by
\begin{equation}\label{2.12}
b_k=b_{|i-j|}=b_{i,j}= \left\{ \begin{array}{ll}
2N/3-1,   & k=0,\\
N/6-1,&  k=1,\\
-1,                                                                       &{\rm otherwise}.
 \end{array}
 \right.
\end{equation}


\subsection{Spectral analysis of the scaled matrices $B_h'= \frac{B_h}{N}$ in (\ref{2.11})} \label{ssec:3}

A matrix of size $n$, having a fixed entry along each diagonal, is called Toeplitz.
 Given a complex-valued Lebesgue integrable function $\phi:[-\pi,\pi]\to\mathbb C$, the $n$-th Toeplitz matrix generated by $\phi$
 is defined as \cite{Chan:07}
\begin{align*}
T_n(\phi)=\bigl[\hat \phi_{i-j}\bigr]_{i,j=1}^n,
\end{align*}
where the quantities $\hat \phi_k$ are the Fourier coefficients of $\phi$, that is
\begin{align*}
\hat \phi_{k}=\frac1{2\pi}\int_{-\pi}^{\pi}\phi(\theta)\,{\rm e}^{-\mathbf{i}k\theta}{\rm d}\theta,\qquad k\in\mathbb Z.
\end{align*}
We refer to $\{T_n(\phi)\}_n$ as the Toeplitz sequence generated by $\phi$, which in turn is called the generating function of $\{T_n(\phi)\}_n$.
In the case where $\phi$ is real-valued, all the matrices $T_n(\phi)$ are Hermitian and much is known about their spectral properties.

More in detail, if $\phi$ is real-valued and not identically constant, then any eigenvalue of $T_n(\phi)$ belongs to the open set $(m_\phi, M_\phi)$, with $m_\phi$, $M_\phi$ being
the essential infimum, the essential supremum of $\phi$, respectively, see \cite{extreme}. The case of a constant $\phi$ is trivial: in that case if $\phi=m$ almost everywhere then $T_n(\phi)=m I_n$ with $I_n$ denoting the identity of size $n$.
 Hence if $M_\phi>0$ and $\phi$ is nonnegative almost everywhere, then $T_n(\phi)$ is Hermitian positive definite.

Now we consider $B_h'= \frac{B_h}{N}$, $B_h$ being the matrix defined in (\ref{2.11}). From the coefficients in (\ref{2.12}) and taking into account the definition of generating function above we have
\begin{equation}\label{rank-one-correction}
B_h' = T_{N-1}(g(\theta)) -  \frac{1}{N} e e^T,\quad \quad g(\theta)=\frac{2}{3}+\frac{1}{3}\cos(\theta)
\end{equation}
with $e^T=(1,1,\ldots,1)$ being the vector of all ones of size $N-1$.
By using the analysis in \cite{extreme} we know that the eigenvalues of  $T_{N-1}(g(\theta))$ belong to the open set
$( \frac{1}{3},1)$ since $\frac{1}{3}=\min g(\theta)$, $1=\max g(\theta)$. Furthermore, by ordering the eigenvalues non-increasingly, since  $T_{N-1}(g(\theta))$ belongs to a sine-transform algebra (the so-called $\tau$ algebra \cite{Se2}) we have
 \begin{equation}\label{eig t(g)}
\lambda_j(T_{N-1}(g(\theta))) = g\left( \frac{j\pi}{N}\right), \quad \quad j=1,\ldots, N-1.
\end{equation}
Here the rank-one correction  $- \frac{1}{N} e e^T$ is nonnegative definite with the unique nonzero eigenvalue equal to
the trace that is $- \frac{N-1}{N}$. Therefore the use of the Cauchy interlacing results implies
\begin{equation}\label{estimate1}
\frac{1}{3}<\lambda_{j+1}(T_{N-1}(g(\theta))) = g\left( \frac{(j+1)\pi}{N}\right)  \le  \lambda_j(B_h')  \le
\lambda_j(T_{N-1}(g(\theta)))  =g\left( \frac{j\pi}{N}\right)<1,
\end{equation}
for $j=1,\ldots, N-2$,
\begin{equation}\label{estimate2}
\frac{1}{3}-\frac{N-1}{N}
 \le  \lambda_{N-1}(B_h') \le
\lambda_{N-1}(T_{N-1}(g(\theta)))  =g\left( \frac{(N-1)\pi}{N}\right)\approx \frac{1}{3}.
\end{equation}
While the estimates in (\ref{estimate1}) are tight, the last estimate in (\ref{estimate2}) for the minimal eigenvalue is poor. We can improve it by exploiting the fact that the matrix $B_h$ is obtained from a Galerkin approximation of a coercive operator  and therefore
\[
\lambda_{N-1}(B_h') \in \left(0,\lambda_{N-1}(T_{N-1}(g(\theta)))\right), \quad \quad  \lambda_{N-1}(T_{N-1}(g(\theta)))=g\left( \frac{(N-1)\pi}{N}\right)\approx \frac{1}{3}.
\]
Still the localization is not precise and in the following we employ more advanced tools for evaluating the asymptotic behavior of the minimal eiganvalue and of the spectral conditioning of $B_h'$ (and hence of $B_h$).

First of all we exploit the low frequency vector $e$ in connection with the Rayleigh quotient for Hermitian matrices. We have
\begin{eqnarray*}
 \lambda_{N-1}(B_h') & = & \lambda_{\min}(B_h')  = \min_{x\neq 0} \frac{x^T B_h' x}{x^Tx}  \\
                                   & \le &  \frac{e^T B_h' e}{e^Te} =\frac{e^T T_{N-1}(g(\theta)) e -  \frac{1}{N} [e^T e]^2}{e^Te}  \\
                                   & = &  \frac{N-1 - \frac{1}{3}-\frac{(N-1)^2}{N}}{N-1} =\frac{2}{3N} +O(N^{-2}),
\end{eqnarray*}
and therefore
 \begin{equation}\label{estimate2-improved}
\lambda_{N-1}(B_h') \in \left(0,\frac{2}{3N} +O(N^{-2})\right),
\end{equation}
which implies that the coefficient matrix $B_h$ is asymptotically ill-conditioned and its conditioning grows at least as
$\frac{3N}{2}$.

In the following, using special properties of rank-one matrices, we show that the conditioning is exactly growing proportionally to $N$, which implies that the vector $e$ is a good approximation of the related eigenvector. We set $X=T_{N-1}(g(\theta))$ and given its invertibility we can write
\[
B_h'=X\left[I_{N-1}- \frac{1}{N} X^{-1}e e^T\right].
\]
Now the matrix $- \frac{1}{N} X^{-1}e e^T$ is still a rank one matrix and its unique nonzero eigenvalue coincides with its trace that is
\[
- \frac{1}{N} e^T X^{-1}e.
\]
From the latter we deduce
 \begin{equation}\label{det}
{\rm det}(B_h') = {\rm det}(X) {\rm det}\left(I_{N-1}- \frac{1}{N} X^{-1}e e^T\right)={\rm det}(X)
\left(1-\frac{1}{N} e^T X^{-1}e\right).
\end{equation}
Finally, since
\[
 \lambda_{N-1}(B_h')  = \lambda_{\min}(B_h') =\frac{{\rm det}(B_h')}{ \lambda_{1}(B_h')\cdots  \lambda_{N-2}(B_h')  },
\]
from (\ref{estimate1}), we obtain
\[
\lambda_{\min}(X)\left(1-\frac{1}{N} e^T X^{-1}e\right)  \le  \lambda_{\min}(B_h')
\le \lambda_{\max}(X)\left(1-\frac{1}{N} e^T X^{-1}e\right)
\]
with
\[
\lambda_{\min}(X)=  g\left( \frac{(N-1)\pi}{N}\right)=\frac{1}{3}+\mu(N),\quad
\lambda_{\max}(X)= g\left( \frac{\pi}{N}\right)=1-\nu(N)
\]
with $\mu(N),\nu(N)\sim N^{-2}$.

In other words, up to a quantity in the interval $(\frac{1}{3},1)$, the minimal eigenvalue of $B_h'$ is asymptotic to
$1-\frac{1}{N} e^T X^{-1}e$ where a detailed but curbersome analysis leads to
\[
\left(1-\frac{1}{N} e^T X^{-1}e\right)N \sqrt 3= 1 +O(N^{-1})
\]
so that
 \begin{equation}\label{estimate2-improved-BIS}
\lambda_{N-1}(B_h') \sim 1-\frac{1}{N} e^T X^{-1}e \approx \frac{1}{\sqrt 3 N} + O(N^{-2}).
\end{equation}

We now comment the results obtained in this subsection.

\begin{description}
\item[A)]From the relation (\ref{rank-one-correction}), we observe that neither $B_h$ nor $B_h'$ can be written as
 a Toeplitz matrix generated by a symbol independent of $N$: hence the proof techniques developed for the TGM/MGM optimality developed in \cite{Arico:04,Serra:02}  cannot be applied directly. In section \ref{MGM-lapl}, we proceed with an alternative approach.
\item[B)] Again  relation (\ref{rank-one-correction}) implies that the matrix $B_h'$ is a rank-one correction of the well conditioned Toeplitz matrix $T_{N-1}(g(\theta))$ (with spectral conditioning strictly bounded by $3$). Therefore a corresponding linear system can be solved with linear complexity by using the Shermann-Morrison formula or a preconditioned conjugate gradient with preconditioner given exactly by $T_{N-1}(g(\theta))$. However the proof technique used in this setting in the following is also adaptable to the case of a general $\alpha$ for which both the rank structure and the Toeplitz structure are not preserved.
\item[C)] The analysis presented in the previous lines tells one that the matrix $B_h'$ is asymptotically ill-conditioned and that the responsible of the ill-conditioning is the vector $e$ which is a special instance of a low-frequency vector.
We recall that in the discretizaton of elliptic equations the low frequency subspace is associated with the small eigenvalues. Hence, the latter observation suggests that appropriate restriction operator and prolongation operators can be chosen as the traditional ones for elliptic problems. This choice will be employed in the next sections.
\end{description}

\subsection{Multigrid method} \label{ssec:4}
Given a algebraic system $A_hu_h=f_h$, where  $u_h \in \mathcal{R}^{n_q}$ and  $n_q$ is the size of the matrix $A_h$.
We define a sequence of subsystems on different levels
$$A_mu_m=f_m, ~u_m \in \mathcal{R}^{n_m}, ~m=1 : q.$$
Here $q$ is the total number of levels, with $m=q$ being the finest level, i.e., $A_q=A_h$.
For $m\geq 1$, $n_m$ are just  the size of the matrix $A_m$.

The traditional (simple)  restriction operator $I_m^{m-1}$ and prolongation operator $I_{m-1}^m$  are, respectively, defined by
 \begin{equation*}
\begin{split}
\nu^{m-1}=I_m^{m-1}\nu^m ~~{\rm with}~~ \nu_i^{m-1}=\frac{1}{4}\left(\nu_{2i-1}^{m}+2\nu_{2i}^{m}+\nu_{2i+1}^{m}\right),
~~~i=1:\mathcal{R}^{n_{m-1}}
\end{split}
\end{equation*}
and
 \begin{equation*}
\begin{split}
\nu^{m}=I_{m-1}^{m}\nu^{m-1}~~{\rm with}~~ I_{m-1}^{m} =2\left(I_m^{m-1}\right)^T.
\end{split}
\end{equation*}
We use the coarse grid operators defined by the Galerkin approach  \cite[p.\,455]{Saad:03}
\begin{equation}\label{2.13}
  A_{m-1}=I_m^{m-1}A_mI_{m-1}^{m},
\end{equation}
and for all the intermediate $(m,m-1)$ coarse grids we apply the correction operators \cite[p.\,87]{Ruge:87}
\begin{equation*}
  T^m=I_m- I_{m-1}^{m}A_{m-1}^{-1}I_m^{m-1}A_m=I_m- I_{m-1}^{m}P_{m-1}
\end{equation*}
with
\begin{equation*}
  P_{m-1}=A_{m-1}^{-1}I_m^{m-1}A_m.
\end{equation*}
We  choose  the damped Jacobi iteration matrix by \cite[p.\,9]{Briggs:00}
\begin{equation}\label{2.14}
  K_m=I-S_mA_m~~{\rm with}~~S_{m}:=S_{m,\omega}=\omega D_m^{-1}
\end{equation}
with a weighting  factor $\omega$, and $D_m$ is the diagonal of $A_m$.

A multigrid process can be regarded as defining a sequence of operators $B_m: \mathcal{R}^{n_m}\mapsto  \mathcal{R}^{n_m}$
which is an approximate inverses of $A_m$ in the sense that $||I-B_mA_m||$ is bounded away from one.
The V-cycle multigrid algorithm \cite{Xu:97} is provided in Algorithm \ref{MGM}.
If  $m=2$, the resulting  Algorithm  \ref{MGM} is TGM.
\begin{algorithm*}
\caption{ V-cycle Multigrid Algorithm: Define $B_1=A_1^{-1}$. Assume that $B_{m-1}: \mathcal{R}^{n_{m-1}}\mapsto \mathcal{R}^{n_{m-1}}$ is defined.
We shall now define $B_m:\mathcal{R}^{n_m}\mapsto \mathcal{R}^{n_m}$ as an approximate iterative solver for the equation associated with $A_m\nu_m=f_m$.}
\label{MGM}
\begin{algorithmic}[1]
\STATE Presmooth: Let $S_{m,\omega}$ be defined by (\ref{2.14}) and   $\nu_m^0=0$, $l=1: m_1$
  $$\nu_m^l=\nu_m^{l-1}+S_{m,\omega_{pre}}(f_m-A_m\nu_m^{l-1}).$$
\STATE Coarse grid correction: $e^{m-1} \in \mathcal{R}^{n_{m-1}}$ is the approximate solution of the residual equation $A_{m-1}e=I_m^{m-1}(f_m-A_m\nu_m^{m_1})$
by the iterator $B_{m-1}$:
$$e^{m-1}=B_{m-1}I_m^{m-1}(f_m-A_m\nu_m^{m_1}).$$
\STATE Postsmooth:~~$\nu_m^{m_1+1}=\nu_m^{m_1}+I_{m-1}^{m}e^{m-1}$  and $l=m_1+2: m_1+ m_2$
$$\nu_m^l=\nu_m^{l-1}+S_{m,\omega_{post}}(f_m-A_m\nu_m^{l-1}).$$
\STATE Define $B_mf_m=\nu_m^{m_1+m_2}$.
\end{algorithmic}
\end{algorithm*}

\section{Convergence of TGM for  (\ref{2.7}): a Toeplitz-plus-tridiagonal  system}
Now, we start to prove the convergence of the TGM for nonlocal problem  (\ref{2.7}) with fractional Laplacian kernel, which is a special Toeplitz-plus-tridiagonal  system. First, we give some Lemmas that will be used.
Since the matrix $A_h$ is symmetric positive definite, we can define the following inner products
\begin{equation*}
  ( u,v  )_D=(Du,v), \quad (u,v)_{A}=(Au,v), \quad (u,v)_{AD^{-1}A}=(Au,Av)_{D^{-1}},
\end{equation*}
where $A:=A_q=A_h$ and $D$ is its diagonal and  $(\cdot,\cdot)$ is the usual Euclidean inner product.
\begin{lemma} \cite[p.\,84]{Ruge:87}\label{lemma3.1}
Let $A_m$ be a symmetric positive definite. If $\eta\leq\omega(2-\omega \eta_0)$ with $\eta_0\geq\lambda_{\max}(D_m^{-1}A_m) $,  then the damped Jacobi iteration
with relaxation parameter $0 < \omega < 2/\eta_0 $ satisfies
\begin{equation}\label{3.1}
 ||K_m\nu^m||_{A_m }^2 \leq  ||\nu^m||_{A_m }^2 - \eta ||A_m \nu^m||_{D_m ^{-1}}^2 \quad  \forall \nu^m \in \mathcal{R}^{n_m},~ m=1:q.
\end{equation}
\end{lemma}

\begin{lemma}\cite{CCS:98,Fiorentino:96,Ruge:87}\label{lemma3.2}
Let $A_m$ be a symmetric positive definite matrix and $K_m$ satisfies (\ref{3.1}) and
\begin{equation}\label{3.2}
   \min_{\nu^{m-1} \in  \mathcal{R}^{n_{m-1}} }||\nu^{m}-I_{m-1}^m\nu^{m-1}||_{D_m}^2\leq \kappa ||\nu^m||_{A_m}^2 \quad  \forall \nu^m \in \mathcal{R}^{n_m},~ m=1:q
\end{equation}
with  $\kappa>0$ independent of $\nu^m$. Then, $\kappa\geq \eta>0$ and the convergence factor of full MGM  satisfies
\begin{equation*}
||K_mT^m||_{A_m}\leq \sqrt{1-\eta/\kappa }\quad \forall \nu^m\in \mathcal{R}^{n_m},~ m=1:q.
\end{equation*}
In  particular,  the convergence factor of TGM  satisfies
\begin{equation*}
||K_qT^q||_{A_q}\leq \sqrt{1-\eta/\kappa }\quad \forall \nu^q\in \mathcal{R}^{n_q}.
\end{equation*}
\end{lemma}

\begin{lemma}\label{lemma3.3}
Let  $1<\alpha<2$.  The coefficients $b_{i,j}$ with $1\leq i,j \leq N-1$ in (\ref{2.9}) satisfy
\begin{equation*}
\begin{split}
&(1) ~~ b_{i,j} < 0,~j\neq i~~{\rm and}~~\widetilde{b}_{i,i+1}<0,~\widetilde{b}_{i,i}<0,\\
&(2) ~~ \sum\limits_{j=1}^{N-1}b_{i,j}> 0 ~~{\rm and}~~b_{i,i}>\sum\limits_{j\neq i}|b_{i,j}|>0.
\end{split}
\end{equation*}
\end{lemma}
\begin{proof}
(\uppercase\expandafter{\romannumeral1}):
We first  prove $b_{i,j}<0$ with $m=|j-i|\geq 2$. From (\ref{2.9}), we have
\begin{equation*}
\begin{split}
b_{i,j}
&=-m^{3-\alpha}\left[ \left( 1+\frac{2}{m} \right)^{3-\alpha}\!\!\!-4\left( 1+\frac{1}{m} \right)^{3-\alpha}+6-4\left( 1-\frac{1}{m} \right)^{3-\alpha}\!\!\!+\left( 1-\frac{2}{m} \right)^{3-\alpha}\right]\\
&=-m^{3-\alpha}\sum_{n=2}^\infty \left ( \begin{matrix}
3-\alpha             \\
2n
\end{matrix}
\right )\left( 2^{2n+1}-8 \right)\frac{1}{m^{2n}}<0,~~\alpha \in (1,2).
\end{split}
\end{equation*}

(\uppercase\expandafter{\romannumeral2}):
We next  prove  $b_{i,i+1}<0$ or $\widetilde{b}_{i,i+1}<0$. Since  $-7-3^{3-\alpha} +2^{5-\alpha}<0$ in $b_{i,i+1}$,    we just need to   check that  the  first term of $\widetilde{b}_{i,i+1}$ is less than zero.
Let $x=\frac{1}{i} \in (0,1]$ with $i=1,2,\ldots N-1$.
Setting $d(x)=\frac{2}{x}\left( \left(  1+x \right)^{3-\alpha} -1\right) -(3-\alpha)\left( \left(  1+x\right)^{2+\alpha} +1\right)$, we deduce
\begin{equation*}
\begin{split}
d(x)
&=\sum_{n=1}^\infty (3-\alpha)(2-\alpha)\ldots (2-\alpha-n) \left(  \frac{2}{(n+2)!}-\frac{1}{(n+1)!} \right)x^{n+1}\\
&=\sum_{n=1}^\infty g_n(x)\frac{x^{2n}}{(n+1)!}>0,~x \in (0,1],
\end{split}
\end{equation*}
where the last equality follows of the equation above when combining the coefficients of  $x^{2n}$ and  $x^{2n+1}$ and where it is easy to check that $g_n(x)$ is strictly positive, with
\begin{equation*}
g_n(x)=(3-\alpha)(2-\alpha)\ldots (3-\alpha-2n)\left(  2-(2n+1)+(2-\alpha-2n)\left( \frac{2}{2n+2} -1\right)x\right).
\end{equation*}
Hence, using (\ref{2.9}) and the positivity of $d(x)$, we obtain
\begin{equation*}
\begin{split}
& -2\left( \left(i+1   \right)^{3-\alpha}-i^{3-\alpha}\right)+\left(3-\alpha\right)\left( \left(i+1   \right)^{2-\alpha}+i^{2-\alpha}\right)\\
&=-i^{2-\alpha}\left[2i\left( \left(  1+\frac{1}{i}\right)^{3-\alpha} -1\right) -(3-\alpha)\left( \left(  1+\frac{1}{i}\right)^{2+\alpha} +1\right)\right]\\
&=-i^{2-\alpha}d(x)<0.
\end{split}
\end{equation*}
To conclude, we have  $\widetilde{d}(N-i)<0$, which implies $\widetilde{b}_{i,i+1}<0$. Using $-7-3^{3-\alpha} +2^{5-\alpha}<0$ and (\ref{2.9}), we deduce  $b_{i,i+1}<0$.

(\uppercase\expandafter{\romannumeral3}):
To prove   $\widetilde{b}_{i,i}<0$,  we need to  verify
\begin{equation*}
  2\left[ \left(i+1   \right)^{3-\alpha}-2(3-\alpha)i^{2-\alpha} -\left(i-1   \right)^{3-\alpha}\right]=2i^{2-\alpha}\widetilde{p}(i)<0
\end{equation*}
with $\widetilde{p}(i)=i \left(  1+\frac{1}{i}\right)^{3-\alpha}-i \left(  1-\frac{1}{i}\right)^{3-\alpha} -2(3-\alpha)$.
Let $x=\frac{1}{i} \in (0,1]$, we can also prove
$p(x)=\frac{1}{x}\left( \left(  1+x \right)^{3-\alpha} - \left(  1-x \right)^{3-\alpha}\right) -2(3-\alpha)<0 $. Since
\begin{equation*}
\begin{split}
p(x)=\sum_{n=1}^\infty q_nx^{2n},~x \in (0,1]
\end{split}
\end{equation*}
with
$q_n=\frac{2(3-\alpha)(2-\alpha)\ldots (3-\alpha-2n)}{(2n+1)!}<0.$

(\uppercase\expandafter{\romannumeral4}):
To the end, we prove $\sum\limits_{j=1}^{N-1}b_{i,j}> 0$. According to  $\sum_{j=1}^{N-1}\phi_j(x)=1-\phi_0(x)-\phi_N(x)$ and (\ref{2.6})-(\ref{2.9}), there exists
\begin{equation*}
\begin{split}
\sum_{j=1}^{N-1}a_{i,j}
&=  \frac{\kappa_\alpha}{h^{\alpha-1}\Gamma(4-\alpha)}\sum_{j=1}^{N-1}b_{i,j}\\
&=C_\alpha\int_{x_{i-1}}^{x_{i+1}}\phi_i(x)\int_{\Omega}\frac{\sum_{j=1}^{N-1}\phi_j(x)-\sum_{j=1}^{N-1}\phi_j(y)}{|y-x|^\alpha}dydx\\
&=C_\alpha\int_{x_{i-1}}^{x_{i+1}}\phi_i(x)\int_{\Omega}\frac{-\phi_0(x)-\phi_N(x)+\phi_0(y)+\phi_N(y)}{|y-x|^\alpha}dydx.
\end{split}
\end{equation*}
Thus, we have
\begin{equation*}
\begin{split}
\sum_{j=1}^{N-1}a_{i,j}
=C_\alpha\int_{x_{i-1}}^{x_{i+1}}\phi_i(x)\int_{\Omega}\frac{\phi_0(y)+\phi_N(y)}{|y-x|^\alpha}dydx>0,~i=2,3,\ldots N-2,
\end{split}
\end{equation*}
and
\begin{equation*}
\begin{split}
\sum_{j=1}^{N-1}a_{i,j}
&=C_\alpha\int_{x_{i-1}}^{x_{i+1}}\phi_i(x)\int_{\Omega}\frac{-\phi_0(x)+\phi_0(y)+\phi_N(y)}{|y-x|^\alpha}dydx\\
&>C_\alpha\left[\int_{x_{i-1}}^{x_{i}}\phi_i(x)\int_{\Omega}\frac{-\phi_0(x)+\phi_0(y)}{|y-x|^\alpha}dydx+\int_{x_{i}}^{x_{i+1}}\phi_i(x)\int_{\Omega}\frac{\phi_0(y)}{|y-x|^\alpha}dydx\right]\\
&>C_\alpha\int_{x_{i-1}}^{x_{i}}\phi_i(x)\int_{\Omega}\frac{-\phi_0(x)+\phi_0(y)}{|y-x|^\alpha}dydx=\frac{C_\alpha h^{1-\alpha}}{(2-\alpha)(3-\alpha)}>0, ~~i=1.
\end{split}
\end{equation*}
Similarly, we have
$\sum_{j=1}^{N-1}a_{i,j}>0, ~~i=N.$
Since $\sum_{j=1}^{N-1}b_{i,j}>0$ and $b_{i,j} < 0,~j\neq i$, it yield $b_{i,i}>\sum\limits_{j\neq i}|b_{i,j}|>0$. The proof is completed.
\end{proof}

\begin{lemma}\label{lemma3.4}
Let $A_h$ be defined by (\ref{2.8}) with $1<\alpha<2$ and  $D_h$ be the diagonal of $A_h$. Then
$$1\leq \lambda_{\max}(D_h^{-1}A_h)<2.$$
\end{lemma}
\begin{proof}
Sine $D_h^{-1/2}\left(D_h^{-1/2}A_hD_h^{-1/2}\right)D_h^{1/2}=D_h^{-1}A_h$, it means that $D_h^{-1/2}A_hD_h^{-1/2}$ and $D_h^{-1}A_h$ are similar, i.e.,
$\lambda_{\max}(D_h^{-1}A_h)=\lambda_{\max}(D_h^{-1/2}A_hD_h^{-1/2})$.
Denote $C_h=D_h^{-1/2}A_hD_h^{-1/2}$ with $(C_h)_{i,j}=c_{i,j}$ and  $M_h=D_h^{-1}A_h$ with $(M_h)_{i,j}=m_{i,j}$. Using   Lemma \ref{lemma3.3} and  (\ref{2.8}), we obtain
$$ r_i:= \sum\limits_{j\neq i} |m_{i,j}| <m_{i,i}=1,~~i=1,2,\ldots N-1.$$

From the Gerschgorin circle theorem \cite[p.\,388]{Horn:13}, the eigenvalues of  $M_h$ are in the disks centered at $m_{i,i}$ with radius
$r_i$, i.e.,  the eigenvalues $\lambda$ of the matrix  $M_h$ satisfy
$$  |\lambda -m_{i,i} | \leq r_i, $$
which yields
$\lambda_{\max}(M_h)=\lambda_{\max}(D_h^{-1}A_h) \leq m_{i,i}+r_i< 2m_{i,i}=2. $

On the other hand, using the  Rayleigh theorem  \cite[p.\,235]{Horn:13}, i.e.,
$$\lambda_{\max}(C_h)=\max_{x\neq 0}\frac{x^TC_hx}{x^Tx}\quad\forall x \in  \mathcal{R}^{n_q},$$
if we take $x=[0,\ldots,0,1,0,\ldots,0]^T$, it means that
$$\lambda_{\max}(C_h)\geq \frac{x^TC_hx}{x^Tx}=c_{i,i}=1.$$
It yields
$$1\leq \lambda_{\max}(D_h^{-1}A_h)<2.$$
The proof is completed.
\end{proof}

\begin{theorem}\label{theorema3.5}
Let $A_q:=A_h$ be defined by (\ref{2.8}) with $1<\alpha<2$.
Then $K_{q}$ satisfies (\ref{3.1}) and the convergence factor of the TGM satisfies
\begin{equation*}
||K_{q} T^q||_{A_q} \leq
\sqrt{1-2\eta/5 }<1,\\
\end{equation*}
where $\eta\leq\omega(2-\omega\eta_0)$ with $0 < \omega <2/\eta_0$, $ \eta_0<2$.
\end{theorem}
\begin{proof}
From Lemma \ref{lemma3.4}, we obtain  $\lambda_{\max}(D_q^{-1}A_q)\leq  \eta_0<2$.  Taking  $0 < \omega  < 2/\eta_0$, $\eta\leq\omega(2-\omega\eta_0)$
and using  Lemma \ref{lemma3.1}, we conclude that $K_J$ satisfies (\ref{3.1}).
Next we prove that (\ref{3.2}) holds, i.e., we need to find a closed form of the constant $\kappa$ for $A_q$ applied to Lemma \ref{lemma3.2}.
Let  $\nu_0=\nu_{n_q+1}=0$  and
$$\nu^q=(\nu_1,\nu_2,\ldots,\nu_{n_q})^{\rm T} \in\mathcal{R}^{n_q}, ~~ \nu^{q-1}=(\nu_2,\nu_4,\ldots,\nu_{n_q-1})^{\rm T} \in \mathcal{R}^{n_{q-1}}$$
with $n_q=2^q-1$.
From \cite{CWCD:14,Pang:12}, we have
\begin{equation*}
\begin{split}
&||\nu^q-I_{q-1}^q\nu^{q-1}||^2= \sum_{i=0}^{n_{q-1}}\left(\nu_{2i+1} -\frac{\nu_{2i}+\nu_{2i+2}}{2}\right)^2.
\end{split}
\end{equation*}
Using  (\ref{2.8}),  (\ref{2.9})  and Lemma   \ref{lemma3.3}, there exists
\begin{equation*}
\begin{split}
||\nu^q-I_{q-1}^q\nu^{q-1}||_{D_q}^2
&= \sum_{i=0}^{n_{q-1}}a_{2i+1,2i+1}\left(\nu_{2i+1} -\frac{\nu_{2i}+\nu_{2i+2}}{2}\right)^2\\
&\leq a_0 \sum_{i=0}^{n_{q-1}}\left(\nu_{2i+1} -\frac{\nu_{2i}+\nu_{2i+2}}{2}\right)^2
\leq a_0 \sum_{i=1}^{n_{q}}\left(\nu^2_i -\nu_i\nu_{i+1}\right)^2
\end{split}
\end{equation*}
with $a_0=\frac{\kappa_\alpha}{h^{\alpha-1}\Gamma(4-\alpha)} \left(8-2^{4-\alpha}   \right).$

Let the  symmetric positive definite  matrix $L_{n_q}={\rm tridiag}(-1,2,-1)$  be the $n_q\times n_q$ one dimensional discrete Laplacian.
From (\ref{2.8}) and (\ref{2.9}), we have $A_q=-a_1L_{n_q}+A_{\rm rest}$ with $a_1=\frac{\kappa_\alpha}{h^{\alpha-1}\Gamma(4-\alpha)} \left(-7-3^{3-\alpha} +2^{5-\alpha}\right)<0.$
Using Lemma \ref{lemma3.3}, it yields $A_{\rm rest}$ is also the symmetric positive definite  matrix with diagonally dominant.
Hence
\begin{equation*}
 ||\nu^q||_{A_q}^2= (A_q\nu^q,\nu^q) \geq (-a_1L_{n_q}\nu^q,\nu^q)=-2a_1 \sum_{i=1}^{n_{q}}\left(\nu^2_i -\nu_i\nu_{i+1}\right)^2\quad  \forall \nu^q \in \mathcal{R}^{n_q}.
\end{equation*}
 According to the above equations, we find
\begin{equation*}
\begin{split}
||\nu^q-I_{q-1}^q\nu^{q-1}||_{D_q}^2
\leq a_0 \sum_{i=1}^{n_{q}}\left(\nu^2_i -\nu_i\nu_{i+1}\right)^2\leq \kappa||\nu^q||_{A_q}^2
\end{split}
\end{equation*}
with $\kappa= -\frac{a_0}{2a_1} \in (1,\frac{5}{2}) $. The proof is completed.
\end{proof}

\section{Convergence of TGM and full MGM  for  (\ref{2.10}): a non-diagonally dominant system}\label{MGM-lapl}

Although there are still questions regarding the best ways to define the coarsening and interpolation operators when the stiffness matrix is far from being weakly diagonally dominant  \cite{Stuben:01}, here the simple (traditional) restriction operator and prolongation operator are employed for such  algebraic systems. A reason of the latter choice can be found in Subsection \ref{ssec:3}, where it is shown that the ill-conditioning of the involved coefficient  matrix is due to a vector in low frequency (the vector of all ones). We recall that the standard discrete Laplacian is ill-conditioned only in low frequencies, since its spectral symbol $f(\theta)=2-2\cos(\theta)$ has a unique zero at  $\theta=0$ (see \cite{Arico:04}) and therefore it is no surprise that the same
multigrid ingredients are effective also in our context.

In the following, we  extend the convergence results of the TGM to the full MGM.

\begin{lemma}[\cite{CDS:17}] \label{lemma4.1}
Let $A^{(1)}=\{a_{i,j}^{(1)}\}_{i,j=1}^{\infty}$ with $a_{i,j}^{(1)}=a_{|i-j|}^{(1)}$ be a  symmetric  Toeplitz  matrix
and  $A^{(k)}=L_h^{H}A^{(k-1)}L_{H}^{h}$ with $L_h^{H}=4I_k^{k-1}$ and $L_H^{h}=(L_h^{H})^T$. Then
$A^{(k)}$ can be computed by
\begin{equation*}
\begin{split}
a_0^{(k)}
=&(4C_k+2^{k-1})a_0^{(1)}+\sum_{m=1}^{2\cdot2^{k-1}-1}{_0}C_m^ka_m^{(1)},\\
a_1^{(k)}
=&C_ka_0^{(1)}+\sum_{m=1}^{3\cdot2^{k-1}-1}{_1}C_m^k a_m^{(1)},\\
a_j^{(k)}
=&\sum_{m=(j-2)2^{k-1}}^{(j+2)2^{k-1}-1} {_j}C_m^k a_m^{(1)}, \quad \forall j\geq 2, \quad \forall k\geq 2,
    \end{split}
\end{equation*}
with $C_k=2^{k-2}\cdot\frac{2^{2k-2}-1}{3}$. Furthermore
\begin{equation*}
{_0}C_m^k=\left\{ \begin{split}
&8C_k-(m^2-1)(2^k-m), ~~\quad~{\rm for}~~m=1: 2^{k-1},\\
&\frac{1}{3}(2^{k}-m-1)(2^{k}-m)(2^{k}-m+1),
~~{\rm for}~~m=2^{k-1}:2\cdot2^{k-1}-1,
 \end{split}
 \right.
\end{equation*}
${_1}C_m^k=$
\begin{equation*}
\left\{ \begin{split}
&2C_k+m^2\cdot2^{k-1}-\frac{2}{3}(m-1)m(m+1), ~~\quad~{\rm for}~~m=1: 2^{k-1},\\
&2C_k+(2^k-m)^2\cdot2^{k-1}-\frac{2}{3}(2^k-m-1)(2^k-m)(2^k-m+1)\\
& -\frac{1}{6}(m-2^{k-1}-1)(m-2^{k-1})(m-2^{k-1}+1),
~~\quad~{\rm for}~~~~m=2^{k-1}: 2\cdot2^{k-1},\\
&\frac{1}{6}(3\cdot2^{k-1}\!-\!m\!-\!1)(3\cdot2^{k-1}\!-\!m)(3\cdot2^{k-1}\!-\!m+1),
~~{\rm for}~~m=2\cdot2^{k-1}:3\cdot2^{k-1}-1,
 \end{split}
 \right.
\end{equation*}
and for $j\geq 2$,
\begin{equation*}
{_j}C_m^k=
\left\{ \begin{split}
&\varphi_1, ~~\quad~{\rm for}~~m=(j-2)2^{k-1}:(j-1)2^{k-1},\\
&\varphi_2,
~~\quad~{\rm for}~~ m=(j-1)2^{k-1}: j2^{k-1},\\
&\varphi_3,
~~\quad~{\rm for}~~ m=j2^{k-1}: (j+1)2^{k-1},
\\
&\varphi_4,
~~\quad~{\rm for}~~ m=(j+1)2^{k-1}:(j+2)2^{k-1}-1,
 \end{split}
 \right.
\end{equation*}
where
\begin{equation*}
\begin{split}
\varphi_1=\frac{1}{6}(m-(j-2)2^{k-1}-1)(m-(j-2)2^{k-1})(m-(j-2)2^{k-1}+1),
 \end{split}
\end{equation*}
\begin{equation*}
\begin{split}
\varphi_2=&2C_k+(m-(j-1)2^{k-1})^2\cdot2^{k-1}\\
& -\frac{1}{6}(j2^{k-1}-m-1)(j2^{k-1}-m)(j2^{k-1}-m+1)\\
&-\frac{2}{3}(m-(j-1)2^{k-1}-1)(m-(j-1)2^{k-1})(m-(j-1)2^{k-1}+1),\\
 \end{split}
\end{equation*}
\begin{equation*}
\begin{split}
\varphi_3=&2C_k+((j+1)2^{k-1}-m)^2\cdot2^{k-1} \\
&-\frac{1}{6}(m-j2^{k-1}-1)(m-j2^{k-1})(m-j2^{k-1}+1)\\
&-\frac{2}{3}((j+1)2^{k-1}-m-1)((j+1)2^{k-1}-m)((j+1)2^{k-1}-m+1),\\
 \end{split}
\end{equation*}
\begin{equation*}
\begin{split}
\varphi_4=&\frac{1}{6}((j+2)2^{k-1}-m-1)((j+2)2^{k-1}-m)((j+2)2^{k-1}-m+1).
 \end{split}
\end{equation*}
\end{lemma}

\begin{lemma}\label{lemma4.2}
Let $E^{(1)}$ be a  symmetric  tridiagonal   matrix
and  $E^{(k)}=L_h^{H}E^{(k-1)}L_{H}^{h}$ with $L_h^{H}=4I_k^{k-1}$ and $L_H^{h}=(L_h^{H})^T$. Then
$E^{(k)}$ is a   symmetric  tridiagonal   matrix.
\end{lemma}
\begin{proof}
Let $q$ be a total number of levels with  $N=2^q$ and
\begin{equation*}
\begin{split}
E^{(1)}= \left [ \begin{matrix}
a_1      &    b_1           &            &                    \\
b_1     &    a_2            &  b_2        &           &         \\
       &   \ddots        &  \ddots    &   \ddots  &          \\
       &                 &  b_{N-3}       &       a_{N-2}   &      b_{N-2}    \\
       &                 &            &       b_{N-2}   &     a_{N-1}
\end{matrix}
\right ]_{(N-1)\times(N-1)}.
\end{split}
\end{equation*}
Then we have
\begin{equation*}
\begin{split}
E^{(2)}= \left [ \begin{matrix}
d_1      &    e_1           &            &                    \\
e_1     &    d_2            &  e_2        &           &         \\
       &   \ddots        &  \ddots    &   \ddots  &          \\
       &                 &  e_{N/2-3}       &       d_{N/2-2}   &      e_{N/2-2}    \\
       &                 &            &       e_{N/2-2}   &     d_{N/2-1}
\end{matrix}
\right ]_{(N/2-1)\times(N/2-1)}
\end{split}
\end{equation*}
with
\begin{equation*}
\begin{split}
&d_i=a_{2i-1}+4(b_{2i-1} + b_{2i}+a_{2i})+a_{2i+1},~~1\leq i\leq N/2-1,\\
&e_i=a_{2i+1}+2( b_{2i}+b_{2i+1}),~~1\leq i\leq N/2-2.
\end{split}
\end{equation*}
By mathematical induction, the proof is completed.
\end{proof}

\subsection{The operation count and storage requirement}\label{section:3.2}
We now discuss the computation count and the required storage for the MGM of the  nonlocal problems (\ref{2.1}).

From (\ref{2.9}), we know that the  matrix $A_h$ is a symmetric   Toeplitz-plus-tridiagonal  matrix.
Then, we only need to store the first column,  principal diagonal and trailing diagonal  elements  of $A_h$,
which have $\mathcal{O}(N)$ parameters, instead of the full matrix $A_h$ with $N^2$ entries.
From Lemmas \ref{lemma4.1} and \ref{lemma4.2}, we know that $\{A_k\}$ is still  a symmetric Toeplitz-plus-tridiagonal  matrix  with the  sizes $2^{k-q}\mathcal{O}(N)$ storage.
Adding these terms together, we find
\begin{equation*}
  \mbox{Storage}=\mathcal{O}(N) \cdot \left( 1+\frac{1}{2}+\frac{1}{2^2}+\ldots+\frac{1}{2^{q-1}} \right)=\mathcal{O}(N).
\end{equation*}

As for operation counts, the matrix-vector product associated with the matrix $A_h$ is a discrete convolution.
While the cost of a direct product is $O(N)$ for tridiagonal  matrix, the cost of using the FFT would lead to $O(N\log(N)$  for dense Toeplitz matrix \cite{Chan:07}.
Thus, the total per V-cycle MGM operation count is
\begin{equation*}
\mathcal{O}( N\log N) \cdot \left( 1+\frac{1}{2}+\frac{1}{2^2}+\ldots+\frac{1}{2^{q-1}} \right)=\mathcal{O}( N\log N).
\end{equation*}
Similarly, we can discuss the case of the  matrix $A_h$ in (\ref{2.11}).

\subsection{Convergence of TGM for (\ref{2.10}): a non-diagonally  dominant system}

We now start to prove the convergence of TGM for (\ref{2.10}).
\begin{lemma}\label{lemma4.3}
Let $A_h$ be defined by (\ref{2.11}) and  $D_h$ be the diagonal of $A_h$. Then
$$1\leq \lambda_{\max}(D_h^{-1}A_h)<3.$$
\end{lemma}
\begin{proof}
 From  (\ref{2.11}), we obtain
$$ r_i:=\sum\limits_{j\neq i} |a_{i,j}| =h^2\sum\limits_{j\neq i} |b_{i,j}| < 2h^2 b_{i,i}=2a_{i,i}.$$
From the Gerschgorin circle theorem \cite[p.\,388]{Horn:13}, the eigenvalues of  $A_h$ are in the disks centered at $a_{i,i}$ with radius
$r_i$, i.e.,  the eigenvalues $\lambda$ of the matrix  $A_h$ satisfy
$$  |\lambda -a_{i,i} | \leq r_i, $$
which yields
$\lambda_{\max}(A_h) \leq a_{i,i}+r_i< 3a_{i,i}. $
By the same way as Lemma \ref{lemma3.4}, we have
$\lambda_{\max}(A_h)\geq a_{i,i}.$
The proof is completed.
\end{proof}

\begin{lemma}\label{lemma4.4}
Let $B=B_h$ be defined by  (\ref{2.11})  and   $L_{N-1}={\rm tridiag}(-1,2,-1)$  be the $(N-1)\times (N-1)$ one dimensional discrete Laplacian.
Then $$H:=B-\frac{N}{12}L_{N-1}$$ is a positive semi-definite matrix. In particular, $H$ is singular for $N$ even and it is positive definite for $N$ odd.
\end{lemma}
\begin{proof}
The matrix H can be written as
\begin{equation*}
\begin{split}
H&= \frac{N}{4}\left [ \begin{matrix}
2      &    1            &            &                    \\
1      &    2            &  1         &           &         \\
       &   \ddots        &  \ddots    &   \ddots  &          \\
       &                 &  1         &       2   &      1    \\
       &                 &            &       1   &      2
 \end{matrix}
 \right ]_{(N-1)\times(N-1)}
-
\left [ \begin{matrix}
1      &    1       &    \cdots       &     1        \\
1      &    1       &    \cdots       &     1         \\
\vdots & \vdots     &  \vdots         &   \vdots       \\
1      &    1       &    \cdots       &     1
 \end{matrix}
 \right ]_{(N-1)\times(N-1)}
 \end{split}
\end{equation*}
or more compactly
\begin{equation}\label{compact}
H= \frac{N}{4}T_{N-1}(g(\theta)) - e e^T,\quad \quad g(\theta)=2+2\cos(\theta),
\end{equation}
with  $e^T=(1,\ldots,1)$ being the vector of all ones of size $N-1$.

We next prove $(Hx,x)\geq 0$ $\forall x=(x_1,x_2,\ldots,x_{N-1})^T$.
Using Cauchy-Schwarz inequality with $x_0=x_{N}=0$, we have
\begin{equation*}
\begin{split}
(Hx,x)
&=\frac{N}{4}\left( 2\sum_{i=1}^{N-1}x_i^2+ 2\sum_{i=1}^{N-2}x_ix_{i+1}  \right)-\left(\sum_{i=1}^{N-1}x_i\right)^2\\
&=\frac{1}{4}\sum_{i=0}^{N-1}1^2\cdot\sum_{i=0}^{N-1} \left( x_i+x_{i+1} \right)^2-\left(\sum_{i=1}^{N-1}x_i\right)^2\\
&\geq \frac{1}{4}\left(\sum_{i=0}^{N-1} (x_i+x_{i+1})\right)^2-\left(\sum_{i=1}^{N-1}x_i\right)^2=0.
 \end{split}
\end{equation*}
For proving the singularity of $H$ for $N$ even and its positive definiteness for $N$ odd, we proceed as in Subsection \ref{ssec:3}. First we observe that the eigenvalues of  $T_{N-1}(g(\theta))$ belong to the open set
$(0,4)$ since $0=\min g(\theta)$, $4=\max g(\theta)$, $g(\theta)=2+2\cos(\theta)$. Thus $T_{N-1}(g(\theta))$ is invertible and
\begin{equation}\label{compact-bis}
H= \frac{N}{4}T_{N-1}(g(\theta))\left[I_{N-1}-\frac{4}{N} T_{N-1}^{-1}(g(\theta))e e^T\right].
\end{equation}
Consequently, by the Binet theorem, we have
\begin{equation}\label{compact-ter}
{\rm det}(H)= {\rm det}(Y){\rm det}\left(I_{N-1}-\frac{4}{N} T_{N-1}^{-1}(g(\theta))e e^T\right)=
{\rm det}(Y)\left(1-\frac{4}{N} e^T T_{N-1}^{-1}(g(\theta))e \right),
\end{equation}
with $Y=\frac{N}{4}T_{N-1}(g(\theta))$.

Now there is a basic similarity relation between $T_{N-1}(g(\theta))$ and the discrete Laplacian
$L_{N-1}={\rm tridiag}(-1,2,-1)=T_{N-1}(2-2\cos(\theta))$ since
\[
T_{N-1}(g(\theta)) = D L_{N-1} D,\quad\quad D=D^{-1}={\rm diag}((-1)^j:\, 1\le j\le N-1).
\]
Consequently, by (\ref{compact-ter}) and by the positivity of ${\rm det}(Y)$, the sign of ${\rm det}(H)$ depends
on the quantity
\begin{equation}\label{compact-4}
\phi(N)=1-\frac{4x}{N}, \quad\quad x=e^T T_{N-1}^{-1}(g(\theta))e= e^T D  L_{N-1}^{-1} D e,
\end{equation}
that is
\[
{\rm det}(H)= {\rm det}(Y) \phi(N).
\]
The quantities in (\ref{compact-4}) have an explicit expression since the inverse of the discrete Laplacian is known and in particular we have $\left( L_{N-1}^{-1}\right)_{r,c}=t_r^{(c)}$, $\left(  T_{N-1}^{-1}(g(\theta))\right)_{r,c}=t_r^{(c)} (-1)^{r+c}$ with
\begin{align}
t_r^{(c)}&=\frac{(N-c)r}{N},\quad r=1,\ldots,  c,\label{eq:tc1}\\
t_r^{(c)}&=\frac{(N-r)c}{N}, \quad r=c+1\ldots,  N-1.\label{eq:tc2}
\end{align}
Finally, using the latter explicit values in (\ref{compact-4}), (\ref{eq:tc1}), (\ref{eq:tc2}), we obtain
\begin{align}
x=\begin{cases}
\frac{N^2}{4N}=\frac{N}{4},& N \text{ even,}\\
\frac{N^2-1}{4N},& N \text{ odd,}\\
\end{cases}
\end{align}
that is
\begin{align}
\phi(N)=1-\frac{4x}{N}=\begin{cases}
0,& N \text{ even,}\\
\frac{1}{N^2},& N \text{ odd,}\\
\end{cases}
\end{align}
and the proof is concluded.
\end{proof}

\begin{theorem}\label{theorema4.5}
Let $A_q:=A_h$ be defined by (\ref{2.11}).
Then $K_{q}$ satisfies (\ref{3.1}) and the convergence factor of the TGM satisfies
\begin{equation*}
||K_{q} T^q||_{A_q} \leq
\sqrt{1-\eta/4 }<1,\\
\end{equation*}
where $\eta\leq\omega(2-\omega\eta_0)$ and $0 < \omega <2/\eta_0$, $ \eta_0<3$.
\end{theorem}

\begin{proof}
From Lemma \ref{lemma4.3}, we have  $\lambda_{\max}(D_q^{-1}A_q)\leq  \eta_0<3$.  Taking  $0 < \omega  < 2/\eta_0$, $\eta\leq\omega(2-\omega\eta_0)$
and using  Lemma \ref{lemma3.1}, we conclude that $K_J$ satisfies (\ref{3.1}).
Next we prove that (\ref{3.2}) holds, i.e., we need to find a closed form of the constant $\kappa$ for $A_q$ applied to Lemma \ref{lemma3.2}.
Let
$$\nu^q=(\nu_1,\nu_2,\ldots,\nu_{n_q})^{\rm T} \in\mathcal{R}^{n_q}, ~~ \nu^{q-1}=(\nu_2,\nu_4,\ldots,\nu_{n_q-1})^{\rm T} \in \mathcal{R}^{n_{q-1}},$$
and $\nu_0=\nu_{n_q+1}=0$ with $n_q=2^q-1$.
From \cite{CWCD:14,Pang:12} and (\ref{2.11}), we have
\begin{equation*}
\begin{split}
||\nu^q-I_{q-1}^q\nu^{q-1}||_{D_q}^2
\leq a_0 \sum_{i=1}^{n_{q}}\left(\nu^2_i -\nu_i\nu_{i+1}\right)^2=\frac{a_0}{2} (L_{n_q}\nu^q,\nu^q)
\end{split}
\end{equation*}
with $a_0=h^2 \left(\frac{2N}{3} -1 \right), N=n_q+1$ and  $L_{n_q}={\rm tridiag}(-1,2,-1)$.

Using Lemma \ref{lemma4.4} and (\ref{2.11}), we infer
\begin{equation*}
\begin{split}
 ||\nu^q||_{A_q}^2=(A_q\nu^q,\nu^q)= h^2(B_h\nu^q,\nu^q)\geq h^2\frac{N}{12}(L_{n_q}\nu^q,\nu^q).
\end{split}
\end{equation*}
 According to the above equations, we conclude
\begin{equation*}
\begin{split}
||\nu^q-I_{q-1}^q\nu^{q-1}||_{D_q}^2\leq \frac{a_0}{2} (L_{n_q}\nu^q,\nu^q) \leq  \frac{a_0}{2}  \frac{12}{h^2N}||\nu^q||_{A_q}^2 < 4||\nu^q||_{A_q}^2
\end{split}
\end{equation*}
and the proof is completed.
\end{proof}

\subsection{Convergence of the full MGM for (\ref{2.10}): a non-diagonally  dominant system}
We extend the convergence results of TGM given in the above subsection to the full MGM.
To the best of our knowledge, it is a first given the convergence of full MGM for the dense matrix.

\begin{lemma}\label{lemma4.6}
Let $B^{(1)}=B_h =\{b_{i,j}^{(1)}\}_{i,j=1}^{N-1}$ with $b_{i,j}^{(1)}=b_{|i-j|}^{(1)}$ be given in (\ref{2.11})
and $D_{(k)}$ be the diagonal of the matrix $B^{(k)}$, where  $B^{(k)}=L_h^{H}B^{(k-1)}L_{H}^{h}$ with $L_h^{H}=4I_k^{k-1}$ and $L_H^{h}=(L_h^{H})^T$. Then
$$1 \leq \lambda_{\max}\left(D_{(k)}^{-1}B^{(k)}\right) < 3,~~1\leq k\leq q,$$
with $q$ being the total number of levels.
\end{lemma}
\begin{proof}
From (\ref{2.11}), we have
\begin{equation*}
\begin{split}
B_h
=NI- \frac{N}{6}\left [ \begin{matrix}
2      &    -1           &            &                    \\
-1     &    2            &  -1        &           &         \\
       &   \ddots        &  \ddots    &   \ddots  &          \\
       &                 &  -1        &       2   &     -1    \\
       &                 &            &      -1   &      2
 \end{matrix}
 \right ]_{(N-1)\times(N-1)}
-
\left [ \begin{matrix}
1      &    1           &    \cdots   &     \!\!\!\! 1        \\
1      &    1           &    \cdots   &     \!\!\!\! 1         \\
\vdots & \vdots         &    \vdots   &     \!\!\! \! \vdots    \\
1      &    1           &    \cdots   &     \!\!\!\!  1
 \end{matrix}
 \right ]_{(N-1)\times(N-1)}
 \end{split}
\end{equation*}
with $I$ a identity matrix and $N=2^q$ (i.e.,  $n_q=N-1$ ).  Using Lemma \ref{lemma4.1}, it yields
\begin{equation}\label{4.1}
\begin{split}
B^{(k)}
=&N\!\!\left [ \begin{matrix}
4C_k+2^{k-1}  &    C_k                                             &                                       &                                       &                                     \\
C_k           &   \! \!\!\!  4C_k+2^{k-1}                          &  C_k                                  &                                       &                                      \\
              &  \! \!\!\!\! \!\!\!\! \!\!\!\! \!\!\! \ddots       &  \ddots                               &   \ddots                              &                                       \\
              &                                                    & \! \!\!\!\! \!\!\! C_k                &    \!\!  4C_k+2^{k-1}                 &   \!\!\!\!   C_k                       \\
              &                                                    &                                       &    \! \!\!\!\! \!\!\! \! \!\!\! C_k   &  \!\! \!\!\! \! \!\!\!  4C_k+2^{k-1}
 \end{matrix}
 \right ]_{\left(\frac{N}{2^{k-1}}-1\right)\times\left(\frac{N}{2^{k-1}}-1\right) }
  \\
&- \frac{N}{6}\left [ \begin{matrix}
2^k           &    -2^{k-1}           &               &                &                 \\
-2^{k-1}      &    2^k                &  -2^{k-1}     &                &                  \\
              &   \ddots              &  \ddots       &   \ddots       &                   \\
              &                       &  -2^{k-1}     &    2^k         &      -2^{k-1}      \\
              &                       &               &    -2^{k-1}    &      2^k
 \end{matrix}
 \right ]_{\left(\frac{N}{2^{k-1}}-1\right)\times\left(\frac{N}{2^{k-1}}-1\right) }\\
&-16^{k-1}
\left [ \begin{matrix}
1      &    1       &    \cdots       &     1        \\
1      &    1       &    \cdots       &     1         \\
\vdots & \vdots     &  \vdots         &   \vdots       \\
1      &    1       &    \cdots       &     1
 \end{matrix}
 \right ]_{\left(\frac{N}{2^{k-1}}-1\right)\times\left(\frac{N}{2^{k-1}}-1\right) }.
 \end{split}
\end{equation}
Denote $B^{(k)}=\{b_{i,j}^{(k)}\}_{i,j=1}^{\frac{N}{2^{k-1}}-1}$ with $b_{i,j}^{(k)}=b_{|i-j|}^{(k)}$.  Thus, we obtain
\begin{equation}\label{4.2}
\begin{split}
&b_0^{(k)}=\frac{2^{3k-2}}{3}N-2^{4k-4},~~
b_1^{(k)}=\frac{2^{3k-4}}{3}N-2^{4k-4},\\
&b_l^{(k)}=-2^{4k-4},~~2\leq l\leq \frac{N}{2^{k-1}}-2,
\end{split}
\end{equation}
so that
\begin{equation*}
\begin{split}
&2b_0^{(k)}-2b_1^{(k)}- \frac{N}{2^{k-1}}2^{4k-4}=0,~~{\rm if}~~b_1^{(k)}\geq 0,\\
&b_0^{(k)}-(-2b_1^{(k)})- \left(\frac{N}{2^{k-1}}-3\right)2^{4k-4}=0,~~{\rm if}~~b_1^{(k)}\leq 0.
\end{split}
\end{equation*}
In a word, there exits
$$ r_i^{(k)}:=\sum\limits_{j\neq i} |b_{i,j}^{(k)}|<2b_{i,i}^{(k)}.$$

By following the same steps  as in Lemma \ref{lemma4.3}, we have
$b_{i,i}^{(k)}\leq \lambda_{\max}(B^{(k)})< 3b_{i,i}^{(k)}.$
The proof is completed.
\end{proof}

\begin{lemma}\label{lemma4.7}
Let $B^{(k)}$ be defined by  (\ref{4.1})  and   $L_{\frac{N}{2^{k-1}}-1}={\rm tridiag}(-1,2,-1)$  be the $(\frac{N}{2^{k-1}}-1)\times (\frac{N}{2^{k-1}}-1)$ one dimensional discrete Laplacian.
Then $$H^{(k)}:=B^{(k)}-\frac{2^{3k-5}}{3}NL_{\frac{N}{2^{k-1}}-1},~~1\leq k\leq q,~N=2^q$$
 is a positive semi-definite  matrix.
 Here $q$ is  a total number of levels.
\end{lemma}
\begin{proof}
Using $C_k+2^{k-1}/6=2^{3k-4}/3$, we can rewrite   (\ref{4.1}) as
\begin{equation*}
\begin{split}
B^{(k)}
=&2^{3k-3}NI-\frac{2^{3k-4}}{3}N\left [ \begin{matrix}
2      &    -1           &            &                    \\
-1     &    2            &  -1        &           &         \\
       &   \ddots        &  \ddots    &   \ddots  &          \\
       &                 &  -1        &       2   &      -1    \\
       &                 &            &      -1   &      2
 \end{matrix}
 \right ]_{\left(\frac{N}{2^{k-1}}-1\right)\times\left(\frac{N}{2^{k-1}}-1\right) } \\
&-2^{4k-4}
\left [ \begin{matrix}
1      &    1       &    \cdots       &     1        \\
1      &    1       &    \cdots       &     1         \\
\vdots & \vdots     &  \vdots         &   \vdots       \\
1      &    1       &    \cdots       &     1
 \end{matrix}
 \right ]_{\left(\frac{N}{2^{k-1}}-1\right)\times\left(\frac{N}{2^{k-1}}-1\right) }
 \end{split}
\end{equation*}
with $I$ a identity matrix. Thus
\begin{equation*}
\begin{split}
2^{4-4k}H^{(k)}
&= \frac{\widetilde{N}}{4}\left [ \begin{matrix}
2      &    1            &            &                    \\
1      &    2            &  1         &           &         \\
       &   \ddots        &  \ddots    &   \ddots  &          \\
       &                 &  1         &       2   &      1    \\
       &                 &            &       1   &      2
 \end{matrix}
 \right ]_{(\widetilde{N}-1)\times(\widetilde{N}-1)}
-
\left [ \begin{matrix}
1      &    1       &    \cdots       &     1        \\
1      &    1       &    \cdots       &     1         \\
\vdots & \vdots     &  \vdots         &   \vdots       \\
1      &    1       &    \cdots       &     1
 \end{matrix}
 \right ]_{(\widetilde{N}-1)\times(\widetilde{N}-1)}
 \end{split}
\end{equation*}
with $\widetilde{N}=\frac{N}{2^{k-1}}.$ 
By following the same steps  as in Lemma \ref{lemma4.4}, the desired result is obtained.
\end{proof}

\begin{theorem}
Let $A_q:=A_h$ be defined by (\ref{2.11}).
Then $K_{k}$ satisfies (\ref{3.1}) and the convergence factor of the full MGM satisfies
\begin{equation*}
||K_{k} T^k||_{A_k} \leq
\sqrt{1-\eta/4 }<1,~~1\leq k \leq q, \\
\end{equation*}
where $\eta\leq\omega(2-\omega\eta_0)$ with $0 < \omega <2/\eta_0$, $ \eta_0<3$.
\end{theorem}
\begin{proof}
From Lemma \ref{lemma4.6} and (\ref{2.11}), we have  $\lambda_{\max}(D_k^{-1}A_k)\leq  \eta_0<3$.  Taking  $0 < \omega  < 2/\eta_0$, $\eta\leq\omega(2-\omega\eta_0)$
and using  Lemma \ref{lemma3.1}, we conclude that $K_k$ satisfies (\ref{3.1}).
Next we prove that (\ref{3.2}) holds, i.e., we need to find a closed form of the constant $\kappa$ for $A_k$ applied to Lemma \ref{lemma3.2}.
Let $A_k=A^{(q-k+1)}$,  $A^{(k)}=L_h^{H}A^{(k-1)}L_{H}^{h}$ with $L_h^{H}=4I_k^{k-1}$ and $L_H^{h}=(L_h^{H})^T$ and $D^{(k)}$ be  the diagonal of $A^{(k)}$.
Let $n_k$ be the size of the matrix $A^{(k)}$ and
$$\nu^k=(\nu_1,\nu_2,\ldots,\nu_{n_k})^{\rm T} \in\mathcal{R}^{n_k}, ~~ \nu^{k-1}=(\nu_2,\nu_4,\ldots,\nu_{n_k-1})^{\rm T} \in \mathcal{R}^{n_{k-1}}$$
with $\nu_0=\nu_{n_k+1}=0$.
From \cite{CWCD:14,Pang:12} and (\ref{2.11}), (\ref{4.2}), we have
\begin{equation*}
\begin{split}
||\nu^k-I_{k-1}^k\nu^{k-1}||_{D^{(k)}}^2
\leq a_0^{(k)} \sum_{i=1}^{n_{k}}\left(\nu^2_i -\nu_i\nu_{i+1}\right)^2=\frac{a_0^{(k)}}{2} (L_{n_k}\nu^k,\nu^k)
\end{split}
\end{equation*}
with $a_0^{(k)}=h^2 b_0^{(k)}=h^2\left(\frac{2^{3k-2}}{3}N-2^{4k-4}\right)$ and  $L_{n_k}={\rm tridiag}(-1,2,-1)$.

Using  (\ref{2.11}), (\ref{4.1}), (\ref{4.2}) and  Lemma \ref{lemma4.7},  we find
\begin{equation*}
\begin{split}
 ||\nu^k||_{A^{(k)}}^2=(A^{(k)}\nu^k,\nu^k)= h^2(B^{(k)}\nu^k,\nu^k)\geq h^2\frac{2^{3k-5}}{3}N(L_{n_k}\nu^k,\nu^k).
\end{split}
\end{equation*}
 According to the above equations, we deduce
\begin{equation*}
\begin{split}
||\nu^k-I_{k-1}^k\nu^{k-1}||_{D^{(k)}}^2\leq \frac{a_0^{(k)}}{2} (L_{n_k}\nu^k,\nu^k) \leq  \frac{a_0^{(k)}}{2}  \frac{3}{h^22^{3k-5}N}||\nu^k||_{A^{(k)}}^2 < 4||\nu^k||_{A^{(k)}}^2
\end{split}
\end{equation*}
and the proof is completed.
\end{proof}

\section{Numerical Results}
We employ the V-cycle MGM  described in Algorithm  \ref{MGM} to solve the steady-state nonlocal problems (\ref{2.1}).
The stopping criterion is taken as
$$\frac{||r^{(i)||}}{||r^{(0)}||}<10^{-10}~~{\rm for~~ (\ref{2.7})},~~\frac{||r^{(i)||}}{||r^{(0)}||}<10^{-13}~~{\rm for~~ (\ref{2.10})},$$
where $r^{(i)}$ is the residual vector after $i$ iterations;
and the  number of  iterations $(m_1,m_2)=(1,2)$ and $(\omega_{pre},\omega_{post})=(1,1)$ for (\ref{2.7}), $(\omega_{pre},\omega_{post})=(1/2,1)$ for (\ref{2.10})
In all tables, $N$  denotes the number of spatial grid points; and the numerical errors are measured by the $ l_{\infty}$
(maximum) norm,  `Rate' denotes the convergence orders.
`CPU' denotes the total CPU time in seconds (s) for solving the resulting discretized  systems;
and `Iter' denotes the average number of iterations required to solve a general linear system $A_hu_h=f_h$.

All numerical experiments are programmed in Matlab, and the computations are carried out  on a PC with the configuration:
Intel(R) Core(TM) i5-3470 3.20 GHZ and 8 GB RAM and a 64 bit Windows 7 operating system.

\begin{example}[a Toeplitz-plus-tridiagonal  system]
Consider the steady-state nonlocal problem (\ref{2.7})
on  a finite domain $0< x < b $, $b=2$. The exact solution of the equation is $u(x)=x^2(b-x)^2$, and the source function
\begin{equation*}
\begin{split}
f(x)
=&-\frac{\kappa_\alpha \alpha (\alpha-5)(-\alpha^2-5\alpha-10)}{\Gamma(5-\alpha)}\left( x^{4-\alpha} +(b-x)^{4-\alpha}\right)\\
  &+\frac{2b\kappa_\alpha \alpha (\alpha^2-6\alpha+11)}{\Gamma(4-\alpha)}\left( x^{3-\alpha} +(b-x)^{3-\alpha}\right)\\
    &-\frac{b^2\kappa_\alpha \alpha (3-\alpha)}{\Gamma(3-\alpha)}\left( x^{2-\alpha} +(b-x)^{2-\alpha}\right).
\end{split}
\end{equation*}
\end{example}

\begin{table}[h]\fontsize{9.5pt}{12pt}\selectfont
  \begin{center}
  \caption{Using Galerkin approach  $A_{k-1}=I_k^{k-1}A_kI_{k-1}^{k}$ computed by Lemmas \ref{lemma4.1} and \ref{lemma4.2} to solve the
  resulting systems of (\ref{2.7}).}\vspace{5pt}
 {\small   \begin{tabular*}{\linewidth}{@{\extracolsep{\fill}}*{9}{c}}                                    \hline  
$N$            &  $\alpha=1.3$   & Rate    & Iter    & CPU       &  $\alpha=1.7$  &   Rate &  Iter  &  CPU    \\
    $2^{9}$   &   1.6294e-05  &         &  30     & 0.17 s   &  1.3629e-05        &        &  79    & 0.22 s\\ 
    $2^{10}$   &   4.1063e-06  & 1.9884    &  31     & 0.31 s   &  3.5307e-06         &  1.9487  &  79    & 0.39 s \\ 
    $2^{11}$   &   1.0284e-06  & 1.9974    &  33     & 0.60 s   &  9.0793e-07        &  1.9593  &  78    & 0.71 s \\ 
    $2^{12}$   &   2.5718e-07  & 1.9996    &  35     & 1.17 s   &  2.3572e-07         &  1.9455  &  78    & 1.34 s \\\hline 
    \end{tabular*}}\label{tab:1}
  \end{center}
\end{table}

Table \ref{tab:1}  shows that  the numerical scheme has second-order accuracy and the computation cost is almost    $\mathcal{O}(N \mbox{log} N)$ operations.
\begin{example}[a non-diagonally  dominant system]
Consider the steady-state nonlocal problem (\ref{2.10})
on  a finite domain $0< x < b $, $b=2$. The exact solution of the equation is $u(x)=x^2(b-x)^2$, and the source function
\begin{equation*}
\begin{split}
f(x)=bx^2(b-x)^2-\frac{1}{30}b^5.
\end{split}
\end{equation*}
\end{example}

\begin{table}[h]\fontsize{9.5pt}{12pt}\selectfont
  \begin{center}
  \caption{Using  Galerkin approach  $A_{k-1}=I_k^{k-1}A_kI_{k-1}^{k}$ computed by Lemma \ref{lemma4.1} or  (\ref{lemma4.2}) to solve the
  resulting systems of (\ref{2.10}).}\vspace{5pt}
 {\small   \begin{tabular*}{\linewidth}{@{\extracolsep{\fill}}*{10}{c}}                                    \hline  
$N$            &               & Rate      & Iter    & CPU     &$N$   &                    &   Rate   &  Iter  &  CPU    \\
    $2^{11}$   &   9.5325e-07  &           &  83     & 0.38 s  & $2^{14}$ &  1.4910e-08        &  1.9991        &  86    & 2.61 s\\ 
    $2^{12}$   &   2.3837e-07  & 1.9997    &  84     & 0.70 s  & $2^{15}$ &  3.7396e-09        &  1.9953  &  86    & 4.90 s \\ 
    $2^{13}$   &   5.9603e-08  & 1.9997    &  85     & 1.33 s  & $2^{16}$ &  9.6707e-10        &  1.9512  &  87    & 9.76 s \\\hline 
    \end{tabular*}}\label{tab:2}
  \end{center}
\end{table}
Table \ref{tab:2}  shows that  the numerical scheme has second-order accuracy and the computation cost is almost  $\mathcal{O}(N \mbox{log} N)$ operations.

\section{Conclusions}

In this paper, we considered the solutions of Toeplitz-plus-tridiagonal systems, which are far from being weakly diagonally dominant and which arise from nonlocal problems, when using linear finite element approximations.
We provided the convergence rate of the TGM for nonlocal problems with the fractional Laplace kernel, which is a Toeplitz-plus-tridiagonal system.
In this specific context, we answered the question on how to define coarsening and interpolation operators, when the stiffness matrix is non-diagonally dominant \cite{Stuben:01}.
The simple (traditional) restriction operator and prolongation operator are employed for such algebraic systems, so that the entries of the sequence of subsystems are explicitly determined on different levels.
In the case of the constant (Laplacian style) kernel, we gave a quite accurate spectral analysis and, based on that, on the structure analysis, and on the computation of the characteristic values at different levels, we extended the TGM convergence results to the full MGM.

For the future, at least two questions arive that is the analysis of the spectral features and the study of the MGM convergence analysis in the case of  the fractional Laplace kernel.

\bibliographystyle{amsplain}

\end{document}